\documentclass[12pt,a4paper]{article}
\usepackage{amssymb}
\usepackage{amsmath}
\usepackage{amsfonts}

\newtheorem{theorem}{Theorem}

\newtheorem{corollary}{Corollary}
\newtheorem{definition}{Definition}

\newenvironment{proof}[1][Proof]{\noindent\textbf{#1.} }{\ \rule{0.0em}{0.0em}}
\input{tcilatex}
\begin{document}

\title{\textbf{Infinite Product of Bicomplex Numbers}}
\author{Chinmay Ghosh \\
%EndAName
Guru Nanak Institute of Technology\\
157/F, Nilgunj Road, Panihati, Sodepur\\
Kolkata-700114, West Bengal, India.\\
E-mail: chinmayarp@gmail.com }
\date{}
\maketitle

\begin{abstract}
{\footnotesize In this article the infinite product of bicomplex numbers is
defined and the convergence and divergence of this product are discussed.}

\end{abstract}

\section{\textbf{Introduction}}

William Rowan Hamilton $\left( 1805-1865\right) $ \cite{2}\ discovered
quarternions in $1843.$ In $1853$ Hamilton introduced the idea of
biquarternions. H. Hankel developed a general system of \textquotedblleft
complex numbers with several real generators\textquotedblright\ in $1967.$
Also the mathematicians like G. Berloty $\left( 1886\right) ,$ E. Study $%
\left( 1889\right) ,\left( 1890\right) $, G. Scheffers $\left( 1891\right)
,\left( 1893\right) $ and C. Segre $\left( 1892\right) $ continued their
study in this area.

In $1882$ Corrado Segre $\left( 1860-1924\right) $ \cite{10} introduced a
new number system called bicomplex numbers. Futagawa $\left( 1928\right)
,\left( 1932\right) $ has published two articles on the theory of functions
of quadruples, which are equivalent to the couples of complex numbers.
Scheffers developed the notion of holomorphic bicomplex function. L.
Tchakalov $\left( 1924\right) $, F. Ringleb $\left( 1933\right) $, T. Takasu 
$\left( 1943\right) $, J. D. Riley $\left( 1953\right) $ and many other
mathematicians applied some ideas from classical complex analysis \cite{1a}
of one complex variables for the study of the holomorphy in bicomplex algebra

Though the two number systems are generalization of complex numbers by four
real numbers, they are essentially different. Quarternions form a
noncommutative division algebra but\ bicomplex numbers \cite{5} form a
commutative ring with divisors of zero. Recently many researchers (\cite{1}, 
\cite{3}, \cite{4}, \cite{6}, \cite{7}, \cite{8} and \cite{9}) found
different (algebraic, geometric, topological and dynamical) properties of
bicomplex numbers.

In this article the infinite product of bicomplex numbers is defined and the
convergence and divergence of this product are discussed.

\section{\textbf{Preliminaries}}

\subsection{Basic definitions}

The set of complex numbers $\mathbb{C}$ forms an algebra generated by two
real numbers and an imaginary element $\mathbf{i}$\ with the property $%
\mathbf{i}^{2}=-1.$\ Similarly duplicating this process by the complex
numbers and three imaginary elements $\mathbf{i}_{1},\mathbf{i}_{2}$ and $%
\mathbf{j}$ governed by the rules:%
\begin{equation*}
\mathbf{i}_{1}^{2}=\mathbf{i}_{2}^{2}=-1,\mathbf{j}^{2}=1
\end{equation*}%
and%
\begin{eqnarray*}
\ \mathbf{i}_{1}\mathbf{i}_{2} &=&\mathbf{i}_{2}\mathbf{i}_{1}=j \\
\mathbf{i}_{1}\mathbf{j} &=&\mathbf{ji}_{1}=-\mathbf{i}_{2} \\
\mathbf{i}_{2}\mathbf{j} &=&\mathbf{ji}_{2}=-\mathbf{i}_{1}.
\end{eqnarray*}%
one can define the bicomplex numbers.

Let us denote $\mathbb{C}(\mathbf{i}_{1})=\left\{ x+\mathbf{i}_{1}y:x,y\in 
\mathbb{R}\right\} ;$ and $\mathbb{C}(\mathbf{i}_{2})=\left\{ x+\mathbf{i}%
_{2}y:x,y\in \mathbb{R}\right\} .$ Each of $\mathbb{C}(\mathbf{i}_{1})$ and $%
\mathbb{C}(\mathbf{i}_{2})$\ is isomorphic to $\mathbb{C}.$

\begin{definition}[Bicomplex numbers]
\cite{6} Bicomplex numbers are defined as%
\begin{equation*}
\mathbb{T}=\{z_{1}+\mathbf{i}_{2}z_{2}:z_{1},z_{2}\in \mathbb{C}(\mathbf{i}%
_{1})\}.
\end{equation*}
\end{definition}

One can identify $\mathbb{C}$\ into $\mathbb{T}$\ in two different ways by
the sets%
\begin{equation*}
\mathbb{C}(\mathbf{i}_{1})=\left\{ z_{1}+\mathbf{i}_{2}0\right\} \subset 
\mathbb{T},
\end{equation*}%
\begin{equation*}
\text{and }\mathbb{C}(\mathbf{i}_{2})=\left\{ z_{1}+\mathbf{i}%
_{2}z_{2}:z_{1},z_{2}\in \mathbb{R}\right\} \subset \mathbb{T}.
\end{equation*}%
\ Although $\mathbb{C}(\mathbf{i}_{1})$ and $\mathbb{C}(\mathbf{i}_{2})$\
are isomorphic to $\mathbb{C}$ they are essentially different in $\mathbb{T}.
$\ 

\begin{definition}[Duplex numbers]
\cite{6} If we put $z_{1}=x$ and $z_{2}=\mathbf{i}_{1}y$ with $x;y\in 
\mathbb{R}$ in $z_{1}+\mathbf{i}_{2}z_{2}$, then we obtain the following
subalgebra of hyperbolic numbers, also called \textit{duplex numbers},%
\begin{equation*}
\mathbb{D}=\{x+\mathbf{j}y:\mathbf{j}^{2}=1,x;y\in \mathbb{R}\}.
\end{equation*}
\end{definition}

\begin{definition}[Bicomplex conjugates]
\cite{6} For bicomplex numbers, there are three possible conjugations. Let $%
w\in \mathbb{T}$ and $z_{1},z_{2}\in \mathbb{C}$ such that $w=z_{1}+\mathbf{i%
}_{2}z_{2}$. Then the three conjugations are defined as:%
\begin{eqnarray*}
w^{\dag _{1}} &=&(z_{1}+\mathbf{i}_{2}z_{2}\mathbf{)}^{\dag _{1}}\mathbf{=}%
\bar{z}_{1}+\mathbf{i}_{2}\bar{z}_{2}, \\
w^{\dag _{2}} &=&(z_{1}+\mathbf{i}_{2}z_{2}\mathbf{)}^{\dag _{2}}\mathbf{=}%
z_{1}-\mathbf{i}_{2}z_{2}, \\
w^{\dag _{3}} &=&(z_{1}+\mathbf{i}_{2}z_{2}\mathbf{)}^{\dag _{3}}\mathbf{=}%
\bar{z}_{1}-\mathbf{i}_{2}\bar{z}_{2}.
\end{eqnarray*}%
where $\overline{z}_{k}$ is the standard complex conjugate of complex
numbers $z_{k}\in \mathbb{C}$.
\end{definition}

\begin{definition}[Bicomplex modulus]
\cite{6} Let $z_{1},z_{2}\in \mathbb{C}$ and $w=z_{1}+\mathbf{i}_{2}z_{2}\in 
\mathbb{T}$, then we have that:%
\begin{eqnarray*}
|w|_{\mathbf{i}_{1}}^{2} &=&w.w^{\dag _{2}}\in \mathbb{C}(\mathbf{i}_{_{1}})
\\
|w|_{\mathbf{i}_{2}}^{2} &=&w.w^{\dag _{1}}\in \mathbb{C}(\mathbf{i}_{_{2}})
\\
|w|_{\mathbf{j}}^{2} &=&w.w^{\dag _{3}}\in \mathbb{D}.
\end{eqnarray*}%
Also the \textit{Euclidean }$\mathbb{R}^{4}$\textit{-norm} defined as%
\begin{equation*}
\left\Vert w\right\Vert =\sqrt{|z_{1}|^{2}+|z_{2}|^{2}}=\sqrt{\func{Re}(|w|_{%
\mathbf{j}}^{2})}.
\end{equation*}
\end{definition}

\begin{definition}[Complex square norm]
\cite{9} The complex (square) norm $CN(w)$ of the bicomplex number $w$ is
the complex number $z_{1}^{2}+z_{2}^{2};$ as $w^{\dag _{2}}\mathbf{=}z_{1}-%
\mathbf{i}_{2}z_{2},$ we can see that $CN(w)=ww^{\dag _{2}}.$ Then a
bicomplex number $w=z_{1}+\mathbf{i}_{2}z_{2}$ is invertible if and only if $%
CN(w)\neq 0.$ Precisely,%
\begin{equation*}
w^{-1}=\frac{w^{\dag _{2}}}{CN(w)}\Rightarrow (z_{1}+\mathbf{i}%
_{2}z_{2})^{-1}=\frac{z_{1}}{z_{1}^{2}+z_{2}^{2}}-\mathbf{i}_{2}\frac{z_{2}}{%
z_{1}^{2}+z_{2}^{2}}.
\end{equation*}
\end{definition}

One should note that $CN(w)$\ is not a real number. Also $CN(w)=0$ if and
only if $z_{2}=\mathbf{i}_{1}z_{1}$\ or $z_{2}=-\mathbf{i}_{1}z_{1}$\ i.e,
if we consider $z_{1}$\ and $z_{2}$\ to be vector then they are of equal
length and perpendicular to each other.

\begin{definition}[Null cone]
\cite{11} The set of all zero divisors is called the \textit{null cone}.
That terminology comes from the fact that when $w$ is written as $z_{1}+%
\mathbf{i}_{2}z_{2},$ zero divisors are such that $z_{1}^{2}+z_{2}^{2}=0.$%
These elements are also called singular (non-invertible) elements, otherwise
it is nonsingular (invertible). The set of all \textit{singular elements} of 
$\mathbb{T}$\ is denoted by $\mathcal{NC}$ or $\mathcal{O}_{2}.$
\end{definition}

\begin{definition}[Idempotent basis]
\cite{11} A bicomplex number $w=z_{1}+\mathbf{i}_{2}z_{2}$ has the following
unique idempotent representation:%
\begin{equation*}
w=z_{1}+\mathbf{i}_{2}z_{2}=(z_{1}-\mathbf{i}_{1}z_{2})\mathbf{e}_{1}+(z_{1}+%
\mathbf{i}_{1}z_{2})\mathbf{e}_{2}
\end{equation*}%
where $\mathbf{e}_{1}=\frac{1+\mathbf{j}}{2}$ and $\mathbf{e}_{2}=\frac{1-%
\mathbf{j}}{2}.$\newline
Also 
\begin{equation*}
x_{1}+x_{2}\mathbf{i}_{1}+x_{3}\mathbf{i}_{2}+x_{4}\mathbf{j}%
=((x_{1}+x_{4})+(x_{2}-x_{3}))\mathbf{e}_{1}+((x_{1}-x_{4})+(x_{2}+x_{3}))%
\mathbf{e}_{2}.
\end{equation*}
\end{definition}

\begin{definition}[Projection]
\cite{11} Two \textit{projection operators }$P_{1}$ and $P_{2}$ are defined
from $\mathbb{T}\longmapsto \mathbb{C}(\mathbf{i}_{_{1}})$ as%
\begin{equation*}
P_{1}(z_{1}+\mathbf{i}_{2}z_{2})=(z_{1}-\mathbf{i}_{1}z_{2})
\end{equation*}%
\begin{equation*}
P_{2}(z_{1}+\mathbf{i}_{2}z_{2})=(z_{1}+\mathbf{i}_{1}z_{2}{}_{1}).
\end{equation*}
\end{definition}

Thus a bicomplex number $w=z_{1}+\mathbf{i}_{2}z_{2}$ is written as 
\begin{equation*}
w=z_{1}+\mathbf{i}_{2}z_{2}=P_{1}(z_{1}+\mathbf{i}_{2}z_{2})\mathbf{e}%
_{1}+P_{2}(z_{1}+\mathbf{i}_{2}z_{2})\mathbf{e}_{2}=P_{1}(w)\mathbf{e}%
_{1}+P_{2}(w)\mathbf{e}_{2}.
\end{equation*}

This representation is very useful as addition, subtraction, multiplication
and division can be done term by term.

\subsection{Trigonometric representation of bicomplex numbers}

If $w=z_{1}+\mathbf{i}_{2}z_{2}\notin \mathcal{NC},$\ then 
\begin{eqnarray*}
w &=&z_{1}+\mathbf{i}_{2}z_{2} \\
&=&\sqrt{z_{1}^{2}+z_{2}^{2}}\left( \frac{z_{1}}{\sqrt{z_{1}^{2}+z_{2}^{2}}}+%
\mathbf{i}_{2}\frac{z_{2}}{\sqrt{z_{1}^{2}+z_{2}^{2}}}\right)  \\
&=&r_{c}\left( \cos \Theta _{c}+\mathbf{i}_{2}\sin \Theta _{c}\right)  \\
&=&r_{c}e^{\mathbf{i}_{2}\Theta _{c}}.
\end{eqnarray*}

where $r_{c}=\sqrt{z_{1}^{2}+z_{2}^{2}}=\sqrt{P_{1}(w)P_{2}(w)}=\left\vert
w\right\vert _{\mathbf{i}_{1}}$ is the complex modulus of $w$ and $\Theta
_{c}$ is the complex argument of $w,$ denoted by $\arg _{\mathbf{i}_{1}}w,$
which is obtained by solving the system 
\begin{eqnarray*}
\cos \Theta _{c} &=&\frac{z_{1}}{\sqrt{z_{1}^{2}+z_{2}^{2}}} \\
\sin \Theta _{c} &=&\frac{z_{2}}{\sqrt{z_{1}^{2}+z_{2}^{2}}}.
\end{eqnarray*}

If we denote by $\Theta _{c0}=\theta _{01}+\mathbf{i}_{1}\theta _{02}$ the
principal value of the complex argument, then 
\begin{equation*}
\arg _{\mathbf{i}_{1}}w=\left\{ \Theta _{c0}+2m\pi :m\in \mathbb{Z}\right\} .
\end{equation*}%
The principal value of the complex argument is denoted by $Arg_{\mathbf{i}%
_{1}}w$.

For a fixed $m\in \mathbb{Z},$ we denote by $\arg _{\mathbf{i}%
_{1},m}w=\Theta _{c0}+2m\pi ,$ the m-th branch of complex argument. Thus 
\begin{equation*}
\arg _{\mathbf{i}_{1}}w=\left\{ \arg _{\mathbf{i}_{1},m}w:m\in \mathbb{Z}%
\right\} .
\end{equation*}%
Also 
\begin{eqnarray*}
\frac{z_{1}}{z_{2}} &=&\frac{\cos \Theta _{c}}{\sin \Theta _{c}}\Rightarrow 
\frac{z_{1}-\mathbf{i}_{1}z_{2}}{z_{1}+\mathbf{i}_{1}z_{2}}=\frac{\cos
\Theta _{c}-\mathbf{i}_{1}\sin \Theta _{c}}{\cos \Theta _{c}+\mathbf{i}%
_{1}\sin \Theta _{c}} \\
&\Rightarrow &\frac{P_{1}(w)}{P_{2}(w)}=e^{-2\mathbf{i}_{1}\Theta _{c}} \\
&\Rightarrow &-2\mathbf{i}_{1}\Theta _{c}=\log \frac{P_{1}(w)}{P_{2}(w)} \\
&\Rightarrow &\Theta _{c}=\mathbf{i}_{1}\log \sqrt{\frac{P_{1}(w)}{P_{2}(w)}}%
.
\end{eqnarray*}

\subsection{Bicomplex exponential function}

\cite{1b} If $w$ is any bicomplex number then the sequence $\left( 1+\frac{w%
}{n}\right) ^{n}$ converges to a bicomplex number denoted by $\exp w$ or $%
e^{w},$ called the bicomplex exponential function, i.e.,%
\begin{equation*}
e^{w}=\lim\limits_{n\rightarrow \infty }\left( 1+\frac{w}{n}\right) ^{n}.
\end{equation*}

If $w=z_{1}+\mathbf{i}_{2}z_{2},$ then we get the bicomplex version of Euler
formula 
\begin{equation*}
e^{w}=e^{z_{1}}\left( \cos z_{2}+\mathbf{i}_{2}\sin z_{2}\right)
=e^{\left\vert w\right\vert _{\mathbf{i}_{1}}}\left( \cos \arg _{\mathbf{i}%
_{1}}w+\mathbf{i}_{2}\sin \arg _{\mathbf{i}_{1}}w\right) .
\end{equation*}

The bicomplex exponential function have the following properties:

\begin{itemize}
\item $e^{z_{1}}$ is the complex modulus and $z_{2}$ is the complex argument
of $e^{w}.$

\item $e^{0}=1.$

\item For any bicomplex number $w,e^{w}\notin \mathcal{NC}.$

\item For any two bicomplex numbers $w_{1},w_{2},e^{\left(
w_{1}+w_{2}\right) }=e^{w_{1}}.e^{w_{2}}.$

\item $e^{\left( \mathbf{i}_{2}\pi \right) }=-1.$

\item $e^{w}=e^{\left( w+2\pi \left( m\mathbf{i}_{1}+n\mathbf{i}_{2}\right)
\right) }$ where $m,n$ are two integers.
\end{itemize}

\subsection{Bicomplex logarithm}

\cite{1b} Let $u=v_{1}+\mathbf{i}_{2}v_{2}$ be a bicomplex number and $w$ be
another bicomplex number such that $e^{u}=w.$ Then $w\notin \mathcal{NC}.$
From the equation $e^{u}=w$ we get 
\begin{equation*}
e^{v_{1}}\left( \cos v_{2}+\mathbf{i}_{2}\sin v_{2}\right) =\left\vert
w\right\vert _{\mathbf{i}_{1}}\left( \cos \arg _{\mathbf{i}_{1}}w+\mathbf{i}%
_{2}\sin \arg _{\mathbf{i}_{1}}w\right) .
\end{equation*}%
It follows that 
\begin{equation*}
v_{1}\in \log \left\vert w\right\vert _{\mathbf{i}_{1}},
\end{equation*}%
the complex logarithm of the complex number $\left\vert w\right\vert _{%
\mathbf{i}_{1}},$ and that 
\begin{equation*}
v_{2}\in \arg _{\mathbf{i}_{1}}w=\left\{ \arg _{\mathbf{i}_{1},m}w:m\in 
\mathbb{Z}\right\} =\left\{ Arg_{\mathbf{i}_{1}}w+2m\pi :m\in \mathbb{Z}%
\right\} .
\end{equation*}%
The bicomplex logarithm of a non-singular bicomplex number $w$ is a set
defined by%
\begin{eqnarray*}
\log w &=&\log \left\vert w\right\vert _{\mathbf{i}_{1}}+\mathbf{i}_{2}\arg
_{\mathbf{i}_{1}}w \\
&=&\log \left\vert \left\vert w\right\vert _{\mathbf{i}_{1}}\right\vert +%
\mathbf{i}_{1}\arg \left\vert w\right\vert _{\mathbf{i}_{1}}+\mathbf{i}%
_{2}\arg _{\mathbf{i}_{1}}w \\
&=&\log \left\vert \left\vert w\right\vert _{\mathbf{i}_{1}}\right\vert +%
\mathbf{i}_{1}\left\{ Arg\left\vert w\right\vert _{\mathbf{i}_{1}}+2m\pi
\right\} +\mathbf{i}_{2}\left\{ Arg_{\mathbf{i}_{1}}w+2n\pi \right\} \\
&=&\log \left\vert \left\vert w\right\vert _{\mathbf{i}_{1}}\right\vert +%
\mathbf{i}_{1}Arg\left\vert w\right\vert _{\mathbf{i}_{1}}+\mathbf{i}%
_{2}Arg_{\mathbf{i}_{1}}w+\left\{ \mathbf{i}_{1}2m\pi +\mathbf{i}_{2}2n\pi
\right\} \\
&=&\left\{ \log _{m,n}w:m,n\in \mathbb{Z}\right\}
\end{eqnarray*}%
where $\log _{m,n}w=\log \left\vert \left\vert w\right\vert _{\mathbf{i}%
_{1}}\right\vert +\mathbf{i}_{1}Arg\left\vert w\right\vert _{\mathbf{i}_{1}}+%
\mathbf{i}_{2}Arg_{\mathbf{i}_{1}}w+\left\{ \mathbf{i}_{1}2m\pi +\mathbf{i}%
_{2}2n\pi \right\} $ is called the $\left( m,n\right) $-th branch of the
bicomplex logarithm of $w.$ The principal branch bicomplex logarithm is
defined by $Logw=\log \left\vert \left\vert w\right\vert _{\mathbf{i}%
_{1}}\right\vert +\mathbf{i}_{1}Arg\left\vert w\right\vert _{\mathbf{i}_{1}}+%
\mathbf{i}_{2}Arg_{\mathbf{i}_{1}}w.$ Although the bicomplex logarithm is
not a (univalued) function, every branch of bicomplex logarithm is a
function defined on $\mathbb{T}-\mathcal{NC}.$ If we write $\log w$ as a
function we always take it to be the principal branch of logarithm unless
otherwise stated.

If the idempotent representation of $w$ is $w=P_{1}(w)\mathbf{e}_{1}+P_{2}(w)%
\mathbf{e}_{2}=w^{\prime }\mathbf{e}_{1}+w^{\prime \prime }\mathbf{e}_{2}$
then from Ringleb Decomposition Theorem \cite{5a} we get 
\begin{equation*}
\log w=\left( \log w^{\prime }\right) \mathbf{e}_{1}+\left( \log w^{\prime
\prime }\right) \mathbf{e}_{2}.
\end{equation*}%
Also if the idempotent representation of $u$ is $u=P_{1}(u)\mathbf{e}%
_{1}+P_{2}(u)\mathbf{e}_{2}=u^{\prime }\mathbf{e}_{1}+u^{\prime \prime }%
\mathbf{e}_{2}$ then the equation $e^{u}=w$ is equivalent to the system of
two complex equations%
\begin{eqnarray*}
e^{u^{\prime }} &=&w^{\prime } \\
e^{u^{\prime \prime }} &=&w^{^{\prime \prime }}
\end{eqnarray*}%
which have the solutions in complex logarithm $\log w^{\prime }$ and $\log
w^{\prime \prime }$ respectively. Therefore%
\begin{eqnarray*}
\log w &=&\log \left\vert w\right\vert _{\mathbf{i}_{1}}+\mathbf{i}_{2}\arg
_{\mathbf{i}_{1}}w \\
&=&\log \sqrt{P_{1}(w)P_{2}(w)}+\mathbf{i}_{2}\left( \mathbf{i}_{1}\log 
\sqrt{\frac{P_{1}(w)}{P_{2}(w)}}\right) \\
&=&\log \sqrt{w^{\prime }w^{\prime \prime }}+\mathbf{i}_{2}\left( \mathbf{i}%
_{1}\log \sqrt{\frac{w^{\prime }}{w^{\prime \prime }}}\right) \\
&=&\frac{1}{2}\log w^{\prime }w^{\prime \prime }+\mathbf{i}_{2}\left( 
\mathbf{i}_{1}\frac{1}{2}\log \frac{w^{\prime }}{w^{\prime \prime }}\right) .
\end{eqnarray*}

\subsection{Power series of bicomplex numbers}

An infinite series of bicomplex numbers $\dsum\limits_{n=0}^{\infty }c_{n}$
is said to be convergent if the sequence of its partial sums $%
S_{n}=\dsum\limits_{k=0}^{n}c_{k}$ is convergent, i.e., for $\varepsilon >0$
there exists integer $N$ such that $\left\Vert S_{n}-S_{m}\right\Vert
<\varepsilon $ for all $m,n>N.$

If $c_{n}=c_{n}^{\prime }\mathbf{e}_{1}+c_{n}^{\prime \prime }\mathbf{e}_{2}$
then $\dsum\limits_{n=0}^{\infty }c_{n}$ is convergent if and only if $%
\dsum\limits_{n=0}^{\infty }c_{n}^{\prime }$ and $\dsum\limits_{n=0}^{\infty
}c_{n}^{\prime \prime }$ are convergent in the ordinary sense.

The series $\dsum\limits_{n=0}^{\infty }c_{n}$ converges to a bicomplex
number $S$ if for $\varepsilon >0$ there exists integer $N$ such that $%
\left\Vert S_{n}-S\right\Vert <\varepsilon $ for all $n>N.$

It can be easily checked that $\dsum\limits_{n=0}^{\infty }c_{n}$ converges
to $S=S_{n}^{\prime }\mathbf{e}_{1}+S_{n}^{\prime \prime }\mathbf{e}_{2}$ if
and only if $\dsum\limits_{n=0}^{\infty }c_{n}^{\prime }$ and $%
\dsum\limits_{n=0}^{\infty }c_{n}^{\prime \prime }$ converge to $%
S_{n}^{\prime }$ and $S_{n}^{\prime \prime }$ respectively.

A power series of bicomplex numbers $\dsum\limits_{n=0}^{\infty }c_{n}w^{n}$
is said to be convergent if and only if $\dsum\limits_{n=0}^{\infty
}c_{n}^{\prime }w^{\prime n}$ and $\dsum\limits_{n=0}^{\infty }c_{n}^{\prime
\prime }w^{\prime \prime n}$ converge as $c_{n}w^{n}=\left( c_{n}^{\prime }%
\mathbf{e}_{1}+c_{n}^{\prime \prime }\mathbf{e}_{2}\right) \left( w^{\prime }%
\mathbf{e}_{1}+w^{\prime \prime }\mathbf{e}_{2}\right) ^{n}=\left(
c_{n}^{\prime }\mathbf{e}_{1}+c_{n}^{\prime \prime }\mathbf{e}_{2}\right)
\left( w^{\prime n}\mathbf{e}_{1}+w^{\prime \prime n}\mathbf{e}_{2}\right)
=c_{n}^{\prime }w^{\prime n}\mathbf{e}_{1}+c_{n}^{\prime \prime }w^{\prime
\prime n}\mathbf{e}_{2}.$

The set of all interior points of the set of points at which a power series
is convergent will be termed the region of convergence of the power series.

The series of bicomplex numbers $\dsum\limits_{n=0}^{\infty }c_{n}$
converges absolutely if $\dsum\limits_{n=0}^{\infty }\left\Vert
c_{n}\right\Vert $ converges \ One can verify the necessary and sufficient
condition for the absolute convergence of the series $\dsum\limits_{n=0}^{%
\infty }c_{n}$ is the absolue convergence of the component series $%
\dsum\limits_{n=0}^{\infty }c_{n}^{\prime }$ and $\dsum\limits_{n=0}^{\infty
}c_{n}^{\prime \prime }.$

\section{\textbf{Infinite product of bicomplex numbers}}

\begin{definition}[Infinite product]
If $\left\{ w_{n}\right\} $ is a sequence of bicomplex numbers and if $%
w=\lim\limits_{n\rightarrow \infty }\dprod\limits_{k=1}^{n}w_{k}$ exists,
then $w$ is the infinite product of the numbers $w_{n}$ and it is denoted by%
\begin{equation*}
w=\dprod\limits_{n=1}^{\infty }w_{n}.
\end{equation*}
\end{definition}

If no one of the numbers $w_{n}\in \mathcal{NC}$ and $w=\dprod%
\limits_{n=1}^{\infty }w_{n}$ exists then $w\notin \mathcal{NC}.$\ Let $%
p_{n}=\dprod\limits_{k=1}^{n}w_{k}$ for $n\geq 1;$ then no $p_{n}\in 
\mathcal{NC}$ and $\frac{p_{n}}{p_{n-1}}=w_{n}.$ Since $w\notin \mathcal{NC}$
and $p_{n}\rightarrow w$ we have that $\lim\limits_{n\rightarrow \infty
}w_{n}=1.$ So the necessary condition for the infinite product $%
\dprod\limits_{n=1}^{\infty }w_{n}$\ to be convergent is that $%
\lim\limits_{n\rightarrow \infty }w_{n}=1$ if no one of the numbers $%
w_{n}\in \mathcal{NC}.$ Observe that for $w_{n}=a\neq 0$ for all $n$ and $%
\left\Vert a\right\Vert <1,$ then $\dprod\limits_{n=1}^{\infty }w_{n}=0$
although $\lim\limits_{n\rightarrow \infty }w_{n}=a\neq 0.$

Henceforth we will discuss the the infinite products $\dprod\limits_{n=1}^{%
\infty }w_{n}$ with the supposition that no one of the numbers $w_{n}\in 
\mathcal{NC},$\ unless mentioned otherwise.

Since $e^{w_{1}+w_{2}}=e^{w_{1}}e^{w_{2}}$ for $w_{1},w_{2}\in \mathbb{T},$
let us discuss the convergence of an infinite product by considering the
series $\dsum\limits_{n=1}^{\infty }\log w_{n}$ where $\log $ is the
principal branch of the bicomplex logarithm. The series $\dsum%
\limits_{n=1}^{\infty }\log w_{n}$\ is meaningful if $\log w_{n}$\ is
meaningful for all $n$. As no one of the numbers $w_{n}\in \mathcal{NC},$ $%
\log w_{n}$ is defined and consequently $\dsum\limits_{n=1}^{\infty }\log
w_{n}$\ is meaningful. Now suppose that the series $\dsum\limits_{n=1}^{%
\infty }\log w_{n}$ converges. If $s_{n}=\dsum\limits_{k=1}^{n}\log w_{k}$
and $s_{n}\rightarrow s$ then $e^{s_{n}}\rightarrow e^{s}.$ But $%
e^{s_{n}}=\dprod\limits_{k=1}^{n}w_{k}$ so that $\dprod\limits_{n=1}^{\infty
}w_{n}$ is convergent to $w=e^{s}\notin \mathcal{NC}.$

\begin{theorem}
\label{t1}Let $w_{n}\notin \mathcal{NC}$ for all $n\geq 1.$ Then $%
\dprod\limits_{n=1}^{\infty }w_{n}$\ converges to a nonsingular number iff
the series $\dsum\limits_{n=1}^{\infty }\log w_{n}$\ converges.
\end{theorem}

\begin{proof}
Let $\dprod\limits_{n=1}^{\infty }w_{n}$\ converges to $w\notin \mathcal{NC}%
. $

Let $p_{n}=\dprod\limits_{k=1}^{n}w_{k}=w_{1}w_{2}....w_{n},w=\left\vert
w\right\vert _{\mathbf{i}_{1}}e^{\mathbf{i}_{2}\Theta _{c0}}$ where $%
\left\vert w\right\vert _{\mathbf{i}_{1}}$ is the complex modulus and $%
\Theta _{c0}$\ is the principal value of the complex argument of $w.$ Also
let $l\left( p_{n}\right) =\log \left\vert p_{n}\right\vert _{\mathbf{i}%
_{1}}+\mathbf{i}_{2}\Theta _{nc}$ where $\left\vert p_{n}\right\vert _{%
\mathbf{i}_{1}}$ is the complex modulus of $p_{n}$ and $\func{Re}\Theta
_{c0}-\pi <\func{Re}\Theta _{nc}\leq \func{Re}\Theta _{c0}+\pi .$

If $s_{n}=\dsum\limits_{k=1}^{n}\log w_{k}=w_{1}+\log w_{2}+.....+\log w_{n}$
then $e^{s_{n}}=w_{1}w_{2}....w_{n}=p_{n}$ so that $s_{n}=\log \left\vert
p_{n}\right\vert _{\mathbf{i}_{1}}+\mathbf{i}_{2}\Theta _{nc}+2\pi \mathbf{i}%
_{2}k_{n}=l\left( p_{n}\right) +2\pi \mathbf{i}_{2}k_{n}$ for some integer $%
k_{n}.$

Since $p_{n}\rightarrow w\notin \mathcal{NC}$ thus $\lim\limits_{n%
\rightarrow 1}w_{n}=1$ and consequently $\lim\limits_{n\rightarrow \infty
}\log w_{n}=0.$ So $s_{n}-s_{n-1}=\log w_{n}\rightarrow 0;$ also $l\left(
p_{n}\right) -l\left( p_{n-1}\right) \rightarrow 0.$ Hence $\left(
k_{n}-k_{n-1}\right) \rightarrow 0$ as $n\rightarrow \infty .$\ Since each $%
k_{n}$ is an integer this gives that there is an $n_{0}$ and a $k$ such that 
$k_{m}=k_{n}=k$ for $m,n\geq n_{0}.$ So $s_{n}\rightarrow l\left( z\right)
+2\pi \mathbf{i}_{2}k;$ that is, the series $\dsum\limits_{n=1}^{\infty
}\log w_{n}$\ converges.

Conversely, assume that $\dsum\limits_{n=1}^{\infty }\log w_{n}$\ converges.
If $s_{n}=\dsum\limits_{k=1}^{n}\log w_{k}$ and $s_{n}\rightarrow s$ then $%
e^{s_{n}}\rightarrow e^{s}.$ But $e^{s_{n}}=\dprod\limits_{k=1}^{n}w_{k}$ so
that $\dprod\limits_{n=1}^{\infty }w_{n}$ converges to $w=e^{s}\notin 
\mathcal{NC}.$\newline
\end{proof}

\begin{itemize}
\item Let $w=w^{\prime }\mathbf{e}_{1}+w^{\prime \prime }\mathbf{e}_{2}$
then 
\begin{equation*}
1+w=\left( \mathbf{e}_{1}+\mathbf{e}_{2}\right) +\left( w^{\prime }\mathbf{e}%
_{1}+w^{\prime \prime }\mathbf{e}_{2}\right) =\left( 1+w^{\prime }\right) 
\mathbf{e}_{1}+\left( 1+w^{\prime \prime }\right) \mathbf{e}_{2}.
\end{equation*}%
Thus for $\left\Vert w\right\Vert <1$ and by Ringleb Decomposition Theorem,%
\begin{eqnarray*}
\log \left( 1+w\right) &=&\log \left( 1+w^{\prime }\right) \mathbf{e}%
_{1}+\log \left( 1+w^{\prime \prime }\right) \mathbf{e}_{2} \\
&=&\left( \dsum\limits_{n=0}^{\infty }\frac{\left( -1\right) ^{n-1}}{n}%
w^{\prime n}\right) \mathbf{e}_{1}+\left( \dsum\limits_{n=0}^{\infty }\frac{%
\left( -1\right) ^{n-1}}{n}w^{\prime \prime n}\right) \mathbf{e}_{1} \\
&=&\dsum\limits_{n=0}^{\infty }\left( \frac{\left( -1\right) ^{n-1}}{n}%
w^{\prime n}\mathbf{e}_{1}+\frac{\left( -1\right) ^{n-1}}{n}w^{\prime \prime
n}\mathbf{e}_{1}\right) \\
&=&\dsum\limits_{n=0}^{\infty }\frac{\left( -1\right) ^{n-1}}{n}w^{n} \\
&=&w-\frac{w^{2}}{2}+\frac{w^{3}}{3}-......\cdot
\end{eqnarray*}
\end{itemize}

If $\left\Vert z\right\Vert <1$ then%
\begin{eqnarray*}
\left\Vert 1-\frac{\log \left( 1+w\right) }{w}\right\Vert &=&\left\Vert 
\frac{1}{2}w-\frac{1}{3}w^{2}+......\right\Vert \\
&\leq &\frac{1}{2}\left( \left\Vert w\right\Vert +\left\Vert w\right\Vert
^{2}+.....\right) \\
&=&\frac{1}{2}\frac{\left\Vert w\right\Vert }{1-\left\Vert w\right\Vert }.
\end{eqnarray*}%
If we further require $\left\Vert w\right\Vert <\frac{1}{2}$ then%
\begin{equation*}
\left\Vert 1-\frac{\log \left( 1+w\right) }{w}\right\Vert \leq \frac{1}{2}.
\end{equation*}%
This gives that for $\left\Vert w\right\Vert <\frac{1}{2}$ 
\begin{equation}
\frac{1}{2}\left\Vert w\right\Vert \leq \left\Vert \log \left( 1+w\right)
\right\Vert \leq \frac{3}{2}\left\Vert w\right\Vert .  \label{1}
\end{equation}

\begin{theorem}
\label{t2} Let $w_{n}=w_{n}^{\prime }\mathbf{e}_{1}+w_{n}^{\prime \prime }%
\mathbf{e}_{2}$ such that $\func{Re}w_{n}^{\prime }>-1$ and $\func{Re}%
w_{n}^{\prime \prime }>-1;$ then the series $\dsum\limits_{n=0}^{\infty
}\log \left( 1+w_{n}\right) $ converges absolutely if and only if the series 
$\dsum\limits_{n=0}^{\infty }w_{n}$ converges absolutely.
\end{theorem}

\begin{proof}
If $\dsum\limits_{n=0}^{\infty }\left\Vert w_{n}\right\Vert $ converges then 
$\dsum\limits_{n=0}^{\infty }\left\vert w_{n}^{\prime }\right\vert $ and $%
\dsum\limits_{n=0}^{\infty }\left\vert w_{n}^{\prime \prime }\right\vert $
both converge. Thus $w_{n}^{\prime }\rightarrow 0$ and $w_{n}^{\prime \prime
}\rightarrow 0.$ So eventually $\left\vert w_{n}^{\prime }\right\vert <\frac{%
1}{2}$ and $\left\vert w_{n}^{\prime \prime }\right\vert <\frac{1}{2}$ and
consequently $\left\Vert w_{n}\right\Vert <\frac{1}{2}.$ By $\left( \ref{1}%
\right) $ $\dsum\limits_{n=0}^{\infty }\left\Vert \log \left( 1+w_{n}\right)
\right\Vert $ is dominated by a convergent series, and it must converge also.

Conversely if $\dsum\limits_{n=0}^{\infty }\left\Vert \log \left(
1+w_{n}\right) \right\Vert $ converges, then it follows that $\left\Vert
w_{n}\right\Vert <\frac{1}{2}$ for sufficiently large $n.$ Again $\left( \ref%
{1}\right) $ allows us to conclude that $\dsum\limits_{n=0}^{\infty
}\left\Vert w_{n}\right\Vert $ converges.\newline
\end{proof}

We wish to define the absolute convergence of the infinite product. Unlike
the absolute convergence of infinite series the convergence of $%
\dprod\limits_{n=1}^{\infty }\left\Vert w_{n}\right\Vert $ does aloow the
convergence of $\dprod\limits_{n=1}^{\infty }w_{n}.$ In fact if $w_{n}=-1$
for all $n;$ then $\left\Vert w_{n}\right\Vert =1$ for all $n$ so that $%
\dprod\limits_{n=1}^{\infty }\left\Vert w_{n}\right\Vert $ converges to $1.$
However $\dprod\limits_{n=1}^{\infty }w_{n}=\pm 1$ depending on whether $n$
is even or odd, so that $\dprod\limits_{n=1}^{\infty }w_{n}$ does not
converge.

\begin{definition}
\label{d1}If $w_{n}\notin \mathcal{NC}$ for all $n$ then the infinite
product $\dprod\limits_{n=1}^{\infty }w_{n}$ is said to converge absolutely
if the series $\dsum\limits_{n=1}^{\infty }\log w_{n}$ converges absolutely.
\end{definition}

According to Theorem $\ref{t1}$ and the fact that the absolute convergence
of a series imples the convergence, we have that absolute convergence of a
product implies the convergence of the product. Also any rearrangement of
the terms of an absolutely convergent product follows the absolutely
convergence. Combining Theorem $\ref{t1}$ and Theorem $\ref{t2}$ with the
Definition $\ref{d1}$ we get the following fundamental criterion for the
convergence of an infinite product.

\begin{corollary}
If $w_{n}\notin \mathcal{NC}$ for all $n$ then the infinite product $%
\dprod\limits_{n=1}^{\infty }w_{n}$ converges absolutely if and only if the
series $\dsum\limits_{n=0}^{\infty }\left( w_{n}-1\right) $ converges
absolutely.
\end{corollary}

\end{document}